\documentclass{amsart}

\usepackage{amsrefs,amssymb,amsmath,amsthm,amsfonts,epsfig,graphicx,psfrag}
\usepackage[dvips]{color}

\newcommand{\cal}{\mathcal}
\newcommand{\R}{\mathbb R}
\newcommand{\Z}{\mathbb Z}

\newcommand{\per}{\text{per}}
\renewcommand{\epsilon}{\varepsilon}
\renewcommand{\phi}{\varphi}

\DeclareMathOperator{\card}{card}
\DeclareMathOperator{\clos}{clos}
\DeclareMathOperator{\interior}{int}
\DeclareMathOperator{\dist}{dist}
\DeclareMathOperator{\W}{W}
\DeclareMathOperator{\qu}{\cal Q} 

\theoremstyle{plain}
\newtheorem{teo}{Theorem}
\newtheorem*{teo*}{Theorem}

\newtheorem{lema}{Lemma}
\newtheorem{prop}{Proposition}

\theoremstyle{definition}
\newtheorem{df}{Definition}

\theoremstyle{remark}
\newtheorem{obs}{Remark}

\begin{document}
\title{Low dimensional bifurcations of snap-back repellors}
\author[A. Artigue]{Alfonso Artigue}
\address{DMEL, Regional Norte, Universidad de la Rep\'ublica, Uruguay}
\email{artigue@unorte.edu.uy}

\begin{abstract} 
We study the relationship between homoclinic orbits associated to repellors, usually called snap-back repellors, and expanding sets 
of smooth endomorphisms. 
Critical homoclinic orbits 
constitutes an interesting bifurcation that is locally 
contained in the boundary of the set of maps having homoclinic orbits. 
This and other possible routes to the creation of homoclinic orbits are considered in low dimensions. 
\end{abstract}

\maketitle

\begin{section}{Introduction}


\par Given a Riemannian manifold $M$, and a positive integer $r$, 
denote by $C^r(M)$ the set of class $C^r$ maps of $M$. 
For $f\in C^r(M)$, a periodic point $x$ of period $k$ is \emph{expanding}, also called a \emph{repellor}, 
if the differential $Df^k_x$ is a linear expanding map. 
In this case, the periodic orbit $\gamma=\{f^j(x): 0\leq j\leq k-1\}$ is also called a repellor. 
A bi-infinite sequence $\{x_n: n\in \Z\}$ is called \emph{an orbit} of 
$f$ if $f(x_{n+1})=x_n$ for every $n\in\Z$. 
An orbit $\{x_n: n\in \Z\}$ is said \emph{homoclinic} to a repellor $\gamma$ if $x_n\to \gamma$ as $n\to\pm\infty$. The only possibility for the existence of a homoclinic orbit associated to a repellor $\gamma$ is that there exists $N\in\Z$ such that $x_N\in\gamma$. Thus this phenomenon is exclusive of noninvertible maps. Note also that any $\{x_n\}$ in the homoclinic orbit is a chain recurrent point, so it is reasonable to expect that under this hypothesis the map has nontrivial dynamics.
    
\par A point $x$ is \emph{singular} or \emph{critical} if $Df_x$ is noninvertible. 
The set of singular points of $f$ will be denoted by $S_f$. 
The homoclinic orbit $\{x_n\}$ is called \emph{regular} if no $x_n$ belongs to $S_f$; otherwise, the orbit is called \emph{critical}. 
The concept of homoclinic orbit associated to a repellor was introduced by Marotto in 1978, 
(see ~\cite{Ma}); he used the name \emph{snap-back repellor} for a homoclinic orbit, and proved that the existence 
of a homoclinic orbit implies chaos in the sense of Devaney; 
in particular, the map has infinitely many periodic orbits. Later, 
Gardini ~\cite{G} gave some examples, studying the bifurcations created by perturbing such an orbit. 
Also Mora \cites{Mo98,Mo}, considered this concept, and found the limits of renormalizations in a neighborhood of such an orbit.
  
Here it will be discussed to what extent this definition of regular 
and critical homoclinic orbits extends the concept of transverse and tangential 
homoclinic orbits for diffeomorphisms. 
Let us now give precise statements of our results.

\par Let $f\in C^1(M)$ have an expanding periodic orbit $\gamma$ of period $k$ and $\{x_n\}$ a regular homoclinic orbit 
associated to $\gamma$. Such a homoclinic orbit gives rise to an expanding set of $f$, 
just as for the case of diffeomorfisms, a transverse homoclinic orbit gives rise to a horseshoe. 
The following result is well known: 

\begin{teo*}
Given any neighborhood $U$ of a periodic repellor $\gamma$ with a regular homoclinic orbit associated there exists a positive integer $m$ and a Cantor set $K\subset U$ such that $K$ is $f^m$ invariant and the restriction of $f^m$ to $K$ is expanding. The dynamics of $f^m$ in $K$ is conjugated to a unilateral shift. 
\end{teo*}

See for example \cites{Ar,G,Rot} for a proof. 
Recall that $K$ is said to be an $f$-\emph{expanding} set if the following conditions hold: 
\begin{enumerate}
\item there exist constants $C>0$ and $\lambda>1$ such that $$ ||Df_x^n(v)||\geq C\lambda^n||v||, $$ for every vector $v$ tangent to $M$ at $x$ and every $n>0$,
\item $K$ is isolated: there exists $U$ a neighborhood of $K$ such that $$K=\cap_{n\geq 0} f^{-n}(U).$$
\end{enumerate}

It is interesting to question to what extent this explains the dynamics 
in a neighborhood of a homoclinic orbit. 
To give a more accurate description of the meaning of this and the following result, 
define, in analogy with the invertible case, the \emph{homoclinic class} 
of a periodic expanding orbit $\gamma$ as the closure 
of the intersection of the set of preorbits of $\gamma$ with the unstable 
set of $\gamma$. A \emph{preorbit} of a point $x_0$ is a sequence $\{x_n\}_{n\geq0}$ such 
that $f(x_{n+1})=x_n$. A preorbit of a cycle is a preorbit of a point of the cycle. 
The \emph{unstable set} of the repellor $\gamma$ is defined as the set of points $W^u_f(\gamma)$ having a preorbit that converges to $\gamma$. The homoclinic class of $\gamma$ under the map $f$ will be denoted by $H_\gamma(f)$.
	
\par As an illustrative example, consider the 
map $z\to z^2$ acting on the circle $S^1$. This 
map has a fixed expanding point at $z_0=1$ whose 
homoclinic class is the whole circle. Beginning 
with a particular homoclinic orbit associated to
 $z_0$, one obtains a Cantor expanding set contained 
in the homoclinic class. So the following question
gain in interest: given an 
expanding set $K$, is it true that it coincides with the 
homoclinic class of an expanding repellor? the answer is 
affirmative if one assures that the set $K$ is 
indecomposable in a dynamical sense, for example, if it is 
transitive (i.e. there is a point whose future orbit is dense).

\begin{teo}\label{class} 
Every uncountable expanding set $K$ has homoclinic orbits associated to periodic repellors. 
If in addition $K$ is transitive then
$$K=\clos(\cup_{n\geq 0}f^{-n}(\gamma)\cap W^u(\gamma)\cap K)$$ 
for every periodic orbit $\gamma\subset K$. 
In particular $K$ is contained in the homoclinic class of $\gamma$. \end{teo}

\par The next problem under study will be the bifurcation a critical homoclinic orbit produces. 

\begin{df} A critical point $c$ of $f\in C^2(M)$ is called a 
\emph{fold type point} if $f$ is locally equivalent to the Whitney cannonical 
form $\qu(a_1,a_2,\ldots,a_d)=(a_1,a_2,\ldots,a_d^2)$, 
where $d$ is the dimension of $M$. This means that there exist neighborhoods 
$U$ of $c$ and $V$ of $f(c)$, and diffeomorphisms 
$\varphi:U\to \R^d$ and $\psi:V\to \R^d$ such that $\psi f=\qu\varphi$ and $\varphi(c)=0$, $\psi(f(c))=0$. 
\end{df}

\par To see more on this subject the reader can consult the text \cite{GG}. Locally at a fold type point $c$, the map $f$ is of type $(2,0)$, meaning that $f(S_f)$ disconnects the neighborhood $V$ of $f(c)$ in such a way that each point in one of the components of the complement of $S_f$ in $V$ has two preimages in $U$ and each point in the other has no preimages in $U$. It is known that there exists an open and dense set $\mathcal G$ of maps in $C^3(M)$ such that, for every $g\in \mathcal G$, the set of fold type points is open and dense in $S_g$.

\par Assume that $f$ is a mapping having a repellor $\gamma$ with a homoclinic orbit $\{x_n\}$ such that $x_0$ is a critical point of fold type of $f$, and every other $x_n$ is regular. This plays the role of a quadratic homoclinic tangency in the context of diffeomorphisms. First it is shown that this constitutes a codimension one phenomenon:

\begin{teo}\label{c+c-} Given any small open set $U$ containing $\gamma\cup\{x_n\}$, there exists a $C^2$ neighborhood $\mathcal U$ of $f$ and a codimension one submanifold $\mathcal S\subset \mathcal U$ such that $\mathcal U\setminus\mathcal S$  is the union of two open sets $\mathcal U_\pm$ such that:
\begin{enumerate} \item Every $g\in \mathcal U_+$ has a regular homoclinic orbit associated to a repellor $\gamma_g$, both contained in $U$.
\item Every $g\in \mathcal U_-$ has no homoclinic orbit contained in $U$.
\item Every map in $\mathcal S$ has a critical homoclinic orbit contained in $U$.
\end{enumerate} 
Furthermore, every map in $\mathcal S$ can be approximated by maps in $\mathcal U$ that have a critical periodic point whose future orbit is contained in $U$.
\end{teo}

\par This provides a primary explanation of the local bifurcation a (generic) critical homoclinic orbit produces. The problem of understanding globally the transition to the creation of a homoclinic orbit is harder and interesting. The simplest and best known example of such a homoclinic tangency is given by the map $x\to 4x(1-x)$ in the interval. It is the minimum parameter $\mu$ for which the family $f_\mu(x)=4\mu x(1-x)$ presents a homoclinic orbit associated to the fixed point $x=0$. We begin the discussion with another intuitive definition.

\par Another way one can define the unstable set of $\gamma$ is taking the union of the future images of a small neighborhood $U$ of $\gamma$ determined by the condition: $\cap_{n\leq 0}f^n(U)=\gamma$. 
It is clear that there exists a homoclinic orbit associated 
to $\gamma$ if and only if 
$(f^{-1}(\gamma)\setminus\gamma)\cap \W^u_f(\gamma)\neq\emptyset$.

\par Now suppose that $f_\mu$ is a continuous one parameter family of maps in $C^1(M)$. Assume that for every $\mu\in[0,1]$ the map $f_\mu$ has a fixed repellor at a point $p\in M$. Suppose also that $f_0$ has at least one regular homoclinic orbit associated to $p$ and that $f_1$ has no homoclinic orbit associated to $p$. Under these hypothesis, define $\mu_0$ as the supremum of the parameters $\mu$ such that $f_\mu$ has a regular homoclinic orbit associated to $p$. 

\par In the one dimensional case, $M=\R$, the unstable set of a fixed repellor is an interval. Let $W$ be the transformation that assigns, to each $\mu\in [0,1]$, the closure of the unstable set of $p$ as a point in the space of closed intervals in $\overline\R=\{-\infty\}\cup\R\cup\{+\infty\}$ considered with its natural topology of $\overline{\R}^2$.

\begin{teo}\label{lastoh} Let $f_\mu:\R\to\R$ a family as above. Then $f_{\mu_0}$ verifies at least one of the following conditions:
\begin{enumerate} \item The transformation $W$ is discontinuous at $\mu_0$, the bifurcation parameter defined above.
\item The map $f_{\mu_0}$ has at least one critical homoclinic orbit associated to $p$, and does not have any regular homoclinic orbit associated to $p$. 
\end{enumerate} 
\end{teo}

Unfortunately, this simple and concrete result is not true in higher dimensions. An example is shown in dimension two.


{\bf Acknowledgement}. The author thanks \'Alvaro Rovella for having proposed the problems considered in the present article and for many enlightening conversations.
This work is related with a graduate monograph under his guidance.
\end{section}

\begin{section}{Proof of Theorem \ref{class}}
\label{capohr}

The first part of Theorem \ref{class} follows by the following:

\begin{prop}\label{lemma8} 
If $K$ is an infinite uncontable expanding set then there exists a homoclinic orbit associated to a periodic repellor. 
\end{prop}

\begin{proof} 
By Theorem II.10 of \cite{C} 
we have that $K=\cup_{n\geq 0}f^{-n}(\clos(\per(f)))$. 
Therefore there is an infinite number of periodic points because $K$ is uncountable.
First we will show that there exist periodic points $p,q$ 
such that $\W^u(p)=\W^u(q)$. 
As proved in \cite{C} there are constants $r>0$, $\mu\in (0,1)$ and $c>0$ such that
\begin{enumerate} 
\item if $x\neq y$ and $f(x)=f(y)$ then $\dist(x,y)>c$ and
\item for every $x\in K$ and $a\in f^{-1}(x)$ there  
exists $\varphi\colon B_r(x)\to K$ such that $\varphi (x)=a$, $f\circ \varphi(y)=y$ 
for every $y\in B_r(x)$ and $\dist (\varphi (z),\varphi(w))\leq \mu\dist (z,w)$, 
for every $z, w\in B_r(x)$.
\end{enumerate}
This implies that for every $p\in K$, $B_r(p)\subset \W^u(p)$. 
Thus if $0<\dist(p,q)< r$ and $p,q\in\per(f)$ we 
have that $q\in\interior(\W^u(p))$. 
\renewcommand{\W}{W^u}
Take $U_q$ a neighborhood of $q$ such that $U_q\subset \W(p)$ and 
$\W(q)=\cup_{n\geq 0}f^n(U_q)$. 
Thus $\W(q)\subset \W(p)$ and analogously we obtain $\W(p)\subset \W(q)$. 
We conclude that $\dist(p,q)<r$ and $p,q\in \per(f)$ implies $\W(q)= \W(p)$. 
Now since $\card(\per)=\infty$ and $K$ is a compact set we have that 
there is an infinite number of pairs of periodic points $(p,q)$ 
whose unstable sets coincide.

\par Taking a power of $f$ we can suppose that there are 
two fixed points $p,q$ such 
that $\W(q)= \W(p)$. 
Let $\{p_n\}_{n\geq0}$ and $\{q_n\}_{n\geq0}$ be such that 
$p_0=p$, $q_0=q$, $p_n\to q$, $q_n\to p$, $f(p_{n+1})=p_n$ and $f(q_{n+1})=q_n$. 
Choose $n_0\geq 0$ such that $q_{n_0}\in B_r(p)$ and an 
open set $U_0$ such that $q_{n_0}\in U_o\subset B_r(p)$. 
Then $f^{n_0}(U_0)$ is a neighborhood of $q$ since $f$ is an open map. 
Let $n_1\geq 0$ be such that $p_{n_1}\in f^{n_0}(U_0)$ and $f^{n_1}(p_{n_1})=p$. 
Take $x\in U_0$ such that $f^{n_0}(x)=p_{n_1}$. 
Since $x\in U_0$ we have that $x\in \W(p)$ and $f^{n_0+n_1}(x)=f^{n_1}(p_{n_1})=p$. And then  $x$ has a homoclinic orbit associated to the periodic repellor $p$. 
\end{proof}

In order to show the second part of Theorem \ref{class} assume that
$K$ is an expanding transitive set.
Observe that if $K$ is finite then it must be a periodic orbit, because transitivity, and the result is trivial. 
We will suppose that $K$ is not a finite set.

\begin{obs}\label{lemma6} 
If $K$ is an expanding transitive set for $f$ then $K$ has no isolated points. 
\end{obs}

\begin{proof}	
Let $x\in K$ whose future orbit is dense in $K$. 
Since $K$ is infinite we have that $x$ is not periodic. 
Notice that every preimage of an isolated point is isolated too, 
this is because $f$ is locally injective. 
If by contradiction we assume that there are isolated points, 
then $x$ has to be itself an isolated point. 
Since $f\colon K\to K$ is onto there is $y\in f^{-1}(x)$ and $y$
has to be an isolated point. Then there exists $n\geq 0$ such that $f^n(x)=y$ and then $x$ is periodic. 
This contradiction proves that there are no isolated points. 
\end{proof}

This remark implies that if $K$ is not a periodic orbit, 
there are homoclinic orbits because if $K$ is an infinite set and has 
no isolated points it is uncountable. 
The following result implies the second part of Theorem \ref{class} taking $p$ as a periodic point.

\begin{prop}
\label{lemma7} 
If $K$ is an expanding transitive set for $f$ then for every 
$p\in K$, $\W^u(p)=K$ and $\cup_{n\geq 0} f^{-n}(p)$ is dense in $K$. 
\end{prop}

\begin{proof} Let $x\in K$ whose future orbit is dense in $K$. 
From the definition of expanding set, it is easy to see that $f\colon K\to K$ is an open map.
Let $r>0$ be such that $B_r(p)\subset \W^u(p)$ and 
suppose that $x\in B_r(p)$. 
Let $\epsilon >0$ be such that $B_\epsilon (x)\subset B_r(p)$. 
Since $K$ is an expanding set there exists $n_0>0$ 
such that for every $n\geq n_0$, 
$B_r(f^n(x))\subset f^n(B_\epsilon (x))$. 
Then $\cup _{n\geq 0}B_r(f^{n}(x))=K$ and then $\W^u(p)=K$. 

On the other hand, 
for every $y\in K$ and $\delta\in (0,r)$ we can 
suppose that $x\in B_{\delta/2}(y)$ and  
$f^{n_1}(x)\in B_{\delta/2}(p)$ for some $n_1\geq 0$. Then there is $z\in B_{\delta/2}(x)$ such that $f^{n_1}(z)=p$ and then $z\in B_\delta(y)$. 
\end{proof}


\end{section}

\begin{section}{Critical Homoclinic Orbit}\label{capohc} In this section we analyze some bifurcations that occur near a critical homoclinic orbit. As was previously explained, generically in $C^2(M)$ the set of fold type point constitute an open and dense subset of the set of critical points. To see what happens when the map $f$ is perturbed, we state without proof the following classical result, whose proof can be found for example in \cite{GG}. 

\begin{prop}\label{truchazo} For every $f\in C^2(M)$ and $x\in M$ a fold type point of $f$, there exists two neighborhoods $U$ of $x$ and $V$ of $p=f(x)$, a local coordinate $\phi\colon V\to \R^d$ and a neighborhood ${\cal U}$ of $f$ in $C^2(M)$ such that the following conditions hold for every $g\in{\cal U}$:
	\begin{enumerate}
    \item Every point in $S_g\cap U$ is a fold type point, where $S_g$ is the set of critical points of $g$.
    \item $S_g\cap U$ and $g(S_g\cap U)$ are codimension one submanifold of $M$.
    \item $\phi(g(S_g\cap U)\cap V)$ is the graph of a function $R_g\in C^1(\R^{d-1},\R)$ such that $g\mapsto R_g$ is differentiable.
    \end{enumerate}
\end{prop}

Let $\phi$, ${\cal U}$ and $R_g$ be as in the previous proposition. 
We can suppose $U$ such that there exist $\psi\colon U\to \R^d$ a local coordinate map and define, for every $g\in {\cal U}$, $\tilde g=\phi\circ g\circ \psi^{-1}$. 
Since $x$ is a fold type point for $f$ we can suppose that $\tilde f(a_1,\dots,a_d)=(a_1,\dots,a_{d-1},a_d^2)$. 
Let ${\cal F}\colon{\cal U}\to\R$ be defined by ${\cal F} (g)= R_g (0)$. 


\begin{lema}\label{l1graf} In the previous notation, the set \[{\cal C}=\{ g\in{\cal U}: p\hbox{ is a critical value of } g|_U\} \] is a codimension one submanifold of ${\cal U}$. \end{lema}

\begin{proof} Let $D_1\subset D_2\subset \R^d$ two disks centered in $0$ and $\rho\colon \R^d\to \R$ a bump function such that $\rho(y)=0$ if $y\notin D_2$, $\rho(y)=1$ if $y\in D_1$ and $\rho(y)\in [0,1]$ otherwise. Consider the one parameter family $f_\mu$ defined by $f_\mu (y)= f(y)$ if $y\notin U$ and $f_\mu(y)=\phi^{-1}\circ q(\psi(y)+\mu\rho(\psi(y))e_d)$ if $y\in U$ (where $e_d=(0,\dots,0,1)\in\R^d$). Observe that $f_0=f$ and that for $\mu$ in a small neighborhood of $0$, $f_\mu\in{\cal U}$. 
    
\par The following shows that $D{\cal F}_f$ is onto. 
\[
	\left.D{\cal F}_f\left(\frac d{d\mu} f_\mu\right|_{\mu= 0}\right)=
	\left.\frac d{d\mu}{\cal F} (f_\mu)\right| _{\mu= 0}=
	\left.\frac d{d\mu}R_{f_\mu} (0)\right|_{\mu= 0}=1
\]
\par The last equality holds because $R_{f_\mu} (0)= \mu$. 
Therefore 0 is a regular value of ${\cal F} $ and ${\cal F}^{-1}(0)$ is a codimension one submanifold of ${\cal U}$. 

\par It remains to show that ${\cal C}={\cal F}^{-1} (0)$ since, by Proposition \ref{truchazo}, ${\cal F}$ is differentiable. 
By the definition of ${\cal F}$ we have that $g\in {\cal F}^{-1}(0)$ if and only if $R_g(0)=0$. And since $\phi (g (S_g\cap U))$ is the graph of $R_g$ we have that $0\in \phi (g (S_g\cap U))$ is equivalent to  $R_g(0)=0$. Then, since $\phi(p)=0$, we conclude $g\in {\cal F}^{-1}(0)$ if and only if $p$ is a critical value of $g|_U$.
\end{proof}

\par The following lemma has the same hypothesis and notation than the
previous one and its proof.

\begin{lema} There exist a neighborhodd of $f$, ${\cal U}_1\subset{\cal U}$ in $C^2(M)$ such that if we define ${\cal C}^+=\{ g\in{\cal U}_1: R_g(0)> 0\}$ and ${\cal C}^-=\{ g\in{\cal U}: R_g (0)< 0\}$ then $p$ is a regular value of $g|_{U}$ if $g\in {\cal C}^\pm$. Moreover, if $g\in{\cal C}^ +$ then $p\notin g(U)$ and if $g\in{\cal C}^- $ then $p\in g(U)$.\end{lema}

\begin{proof} Consider for all $g\in{\cal U}$ the coordinates of $\tilde g$ given by $$\tilde g(a_1,\dots,a_d)=(\tilde g_1(a_1,\dots,a_d),\dots,\tilde g_n(a_1,\dots,a_d)).$$ Our assumptions imply that $\tilde f_d(a_1,\dots,a_d)=a_d^2$ and that if ${\cal U}_1\subset {\cal U}$ is small enough we have that for every $g\in {\cal U}_1$ the coordinate $\tilde g_d$ satisfies the following conditions: (1) if we restrict $\tilde g_d$ to the line determined by fixing $a_i$, for $i=1,\dots,n-1$, we have that $\tilde g_n$ has a minimum value that is reached at the graph of $R_g$; (2) the image of $\tilde g$ is $\{ (a_1,\dots,a_d): a_d\geq R_g(a_1,\dots,a_{d-1}) \}$.

\par Then if $g\in{\cal C}^+$ ($R_g (0)> 0$), 0 does not have preimage by $\tilde g$. And then $p$ does not have preimage in $U$, and that is why it is a regular value of $g|_U$.
\par Suppose $g\in{\cal C}^- $. We have that $R_g (0)< 0$ and 0 is a regular value having two preimages by $\tilde g$, and then $p\in g(U)$ is a regular value. \end{proof}

\begin{proof}[Proof of Theorem \ref{c+c-}] First, we can suppose that there exists a nieghborhood $U_1$ of the repellor fixed point of $f$, $p$, such that for every $g\in {\cal U}_1$ (the neighborhood of $f$ given by the previous lemma) there is a repellor fixed point of $g$, $p_g$, and $U_1\subset \W_g(p_g)$. Furthermore, conjugating with a diffeomorphism of $M$, near the identity map of $M$, we can suppose that $p_g=p$ for every $g\in {\cal U}_1$.

\par Consider a neighborhood ${\cal U}_2 \subset {\cal U}_1 $ of $f$ in $C^2(M)$ such that for every $g\in {\cal U}_2$ and every $y\in U$ (the neighborhood of $x$) there exists a regular sequence of preorbits of $y$ that converges to $p$. Then for every $g\in {\cal U}^-$ we have that $p$ has a regular preimage in $U$, by the previous lemma, and then it has a regular sequence of preorbits that converges to $p$. Then $p$ has a regular homoclinic orbit associated and it is close to the one associated to $f$. In case $g\in {\cal U}^+$ we have that $p$ does not have any preimage by $g$ in $U$, again by the previous lemma, and then there is no homoclinic orbit associated to $p$ close to the one of $f$. If $g\in {\cal C}$ we have that $p$ has a critical fold type preimage in $U$ and by the same arguments we have that $p$ has a homoclinic orbit associated to the one of $f$ that is critical.

\par For the last assertion of the theorem consider the one parameter family $f_\mu\in C^ 2 (M)$ defined in the proof of lemma \ref{l1graf}. We recall that 
\begin{equation}\label{eqderivada}
	\left.\frac d{ d\mu}R_{f_\mu} (0)\right|_{\mu= 0}\neq 0.
\end{equation}
We will prove that for every $\varepsilon>0$ there exists $\mu$, $|\mu|<\varepsilon$, such that there exists a critical periodic point for $f_\mu$ in $U$.

\par Let $G\colon\R^{d-1}\times \R\to\R^{d-1}\times \R$ defined by \[G(w,\mu)= (w, R_{f_\mu} (w)).\] 
The map $G$ is differentiable by Proposition \ref{truchazo}. Observe that $\phi^{-1}\circ G (w, \mu) $ is a critical value of $f_\mu$, and it has a fold type preimage in $U$, by $f_\mu$, that we call $c(w,\mu)$. By  equation (\ref{eqderivada}) we have that $DG_0$ is invertible and by the inverse function theorem $G$ is a local diffeomorphism. Observe that $G (0)= 0$. 

\par Let $K\times[-\epsilon,\epsilon]$ a neighborhood of $0\in \R^d$ and $L\subset V$ a neighborhood of $p$ both homeomorphic to compact balls of $\R^d$ such that $G\colon K\times[-\epsilon,\epsilon]\to \phi(L)$ is a diffeomorphism. 
Let $q\in L$ a point in the homoclinic orbit of $f$ and let $M\subset L$ a neighborhood of $q$ and $m\geq 0$ such that for every $\mu\in [-\epsilon,\epsilon] $, $U\subset f_\mu^ m (M) $ and $f_\mu^ m\colon M\to f_\mu^m(M)$ is a diffeomorphism. Let \[ H\colon K\times[-\epsilon,\epsilon]\to K \times[-\epsilon,\epsilon] \]   defined by 
\[ 
H (w,\mu)= G^{-1}\circ \phi\circ (f_\mu^{ m}| _M)^{- 1}\circ c (w, \mu) 
\] 
The map $c$ is continuous by the way the family $f_\mu$ was defined. Then, $H$  is  continuous since it is the composition of continuous functions. Furthermore its domain was chosen to be homeomorphic to a compact ball of $\R^d$. Then by Brower Theorem we have a fixed point $(z,u)$. We will prove that $\phi^{-1}\circ G(z,u)$ is a critical periodic point. Since $(z,u)$ is a fixed point for $H$ have that \[ G^{-1}\circ \phi\circ (f_\mu^{m}| _M)^{-1}\circ c (z, u)= (z,u) \] and then
\[ c(z,u)= f_\mu^ m\circ \phi^{-1}\circ G(z,u) \]

\par But, by the way the function $c$ was defined, $c(z,u)$ is a critical preimage of $\phi^{-1}\circ G(z,u)$ for $f_u$, that is $f_u\circ c(z,u)= \phi^{-1}\circ G(z,u)$. Then \[f_u^{m+1} (\phi^{-1}\circ G(z,u))= \phi^{-1}\circ G(z,u)\] Then we have the fold type periodic point. \end{proof}

\end{section}    

\begin{section}{One Dimensional Bifurcations} \label{dinuni}

We will consider a family of maps $f_\mu\colon \R\to\R$ such that for every $\mu\in\R$ the maps have a 
repellor fixed point at $p$. 
In order to prove Theorem \ref{lastoh} we need two lemmas.
\renewcommand{\W}{W^u}
\begin{lema}
\label{nopreimagenborde} 
Let $f\colon \R\to\R$ be a $C^1$ map and $p\in\R$ a repellor fixed point. 
Then if $p$ does not have any homoclinic orbit associated then $p$ has no preimage in $\partial \W(p)$. 
\end{lema}

\begin{proof}  We will divide the proof according to the structure of $\W(p)$.
First suppose that $\W(p)$ is bounded. Let $r\in\partial \W(p)\setminus \W(p)$. 
It always holds that $f\colon \W(p)\to \W(p)$ is onto. 
Thus $f\colon\clos(\W(p))\to \clos(\W(p))$ is also onto since $f(\clos(\W(p))) $ is compact. 
Then $r$ has a preimage in $\clos(\W(p))$, suppose $f(s)=r$. Since $r\notin \W(p)$ we have that $s\notin \W(p)$.
Consider the following two cases.
\begin{enumerate} 
\item If $\W(p)$ is half-open then $s= r$. Therefore $r$ is a fixed point and not a preimage of $p$. 
\item Suppose that $\W(p)$ is an open interval. If $r= s$ then $r$ is not a preimage of $p$ as in the previous case. 
On the other hand, $\W(p)=(r,s)$ or $\W(p)=(s,r)$, and we have that $f (r)= s$ and $r$ is periodic of period 2. This is because $f(\W_\pm(p))= \W_\mp(p)$  since there is no preimages of $p$ in $\W(p)$ except for $p$, where $\W_{+}(p)=\{ x\in \W(p): x> p\} $ and  $\W_{-}(p)=\{ x\in \W(p): x< p\} $. 
Recall that $r$ is not a preimage of $p$. 
\end{enumerate} 

Now suppose that $\W(p)$ is not bounded. 
If $\W(p)=\R$ there is nothing to prove. 
Suppose that $\W(p)=  (r,+\infty)$ 
(the same goes for $\W(p)=(-\infty,r)$). 
The number $f'(p)$ cannot be negative because  $f([r,p])$ 
would have to be $[p,+\infty)$ that is not compact. 
Then $f'(p)>1$ and we proceed as before, 
$f\colon[r,p]\to[r,p]$ is onto and $r$ does not have preimage in $(r,p] $, therefore $f(r)=r$. Thus $r$ is not a preimage of $p$.
\end{proof}

\begin{lema}\label{estnooh} If $f$ is a point of continuity of $\W(p)$ and $f$ does not have any homoclinic orbit associated to $p$ then there exist a $C^1$-neighborhood $\cal V$ of $f$ such that every $g\in\cal V$ has no   homoclinic orbit associated to $p$. \end{lema}

\begin{proof} 
Since $p$ is a repellor fixed point we have that for every 
$C^1$-close map there is a repellor fixed point close to $p$.  
Without loss of generality we suppose that this fixed point is $p$.

We know that $\W(p)\cap f^{-1} (p)=\{ p\}$ and 
by Lemma \ref{nopreimagenborde} $\clos(\W(p))\cap f^{-1} (p)=\{ p\}$. 
Let $U_1, U_2\subset \R$ be disjoint neighborhoods of 
$f^{-1} (p)\setminus p$ and $\clos(\W(p))$ respectively. 
Let ${\cal U} _1$ a $C^ 1$-neighborhood of $f$ such that for 
every $g\in{\cal U} _1$,  $g(\clos(\W_g(p)))\subset U_1$ 
($f$ is a point of continuity for $\W(p)$). 
Let ${\cal U} _2$ a $C^ 0$-neighborhood of $f$ such that 
every $g\in{\cal U} _2$ doesn't have preimages of $p$ 
in $U_2$, excepting $p$. 
Then ${\cal V}={\cal U}  _1\cap{\cal U} _2$ satisfies the thesis of the lemma. 
\end{proof}

\begin{proof}[Proof of Theorem \ref{lastoh}] 
Suppose that $f_{\mu_0}$ is a point of continuity 
of $\W(p)$. If $f_{\mu_0} $ has no homoclinic 
orbit associated to $p$ then by Lemma \ref{estnooh} there 
is a $C^1$-neighborhood ${\cal V} $ of  $f$ such 
that every $g\in{\cal V}$ has no homoclinic orbit associated to $p$. 
Since the map $F\colon\R\to C^1(\R)$, 
given by $F(\mu)=f_\mu$ is continuous, 
then $F^{-1} ({\cal V}) $ is a neighborhood of $\mu_0$ and 
then for every   $\mu\in F^{-1}({\cal V}) $, $f_\mu$ has 
no homoclinic orbit associated to $p$, which is a contradiction. 
This proves that $f_{\mu_0}$ has homoclinic orbits. 
The condition: $f_{\mu_0}$ has a regular homoclinic orbit 
associated to $p$, is an open condition, so again we reach a 
contradiction. Then $f_{\mu_0}$ has at least one 
critical homoclinic orbit associated to $p$ and has not regular homoclinic 
orbits.
\end{proof}

\end{section}

\begin{section}{A Two Dimensional Bifurcation}
\renewcommand{\W}{W^u}
Let us give an example of a one parameter family of
 maps $f_\mu\colon\R^2\to\R^2$ such that the bifurcation 
map $f_{\mu_0}$ has no homoclinic orbit associated to the 
repellor fixed point $p=0$ and the unstable set of $p$ is 
the same for every value of the parameter $\mu$. 
The construction of this example consist on three parts. 
(1) We define a set of maps ${\cal U}\subset C^1(\R^2)$ 
where the one parameter family will be found. 
(2) We show that $\W(p)$ does not depend on $\mu$.  
(3) Construct the one parameter family with the desired properties.

\par Let $S^1=\{(x,y)\in \R^2:x^2+y^2=1\}$ and $j \colon S^1\to S^1$ be a diffeomorphism such 
that $I=S^1\cap\{(x,y)\in\R^2: x> 1/ 2\} $ is a  wandering interval of $j$. Let $g\colon\R^+\cup \{ 0\}\to\R^+\cup\{ 0\}$ 
differentiable, such that $0$ is a repellor fixed point and $1$ is an attractor. 
Let $R\colon\R^ 2\to\R^ 2$ the map defined in polar coordinates by $R (\rho,\theta)= (g (\rho), j (\theta)) $. 
This map has a repellor fixed point at $0$ and $S^1$ is an invariant attractor. 
If we define $D=\{(x, y)\in\R^ 2: x^ 2+ y^ 2< 1\} $ then $R (D)= D$. 
Let $\hat I\subset D$ the angle with vertex in the origin generated by $I$. 
In this way $I\setminus \{0\}$ is a wandering set for $R$. The fact that $R$ may not be differentiable at $0$ will be considered later.

\par Let $F\colon\R^ 2\to\R^ 2$ a fold type map such that:
\begin{itemize} \item $F (x, y)= (x, y) $ if $x\leq 1/ 2$
\item $F (D)\subset D$
\item The straight line $x= 3/ 4$ is the set of critical points of $F$ and $F(3/4,y)=(3/4,y)$ if $(3/4,y)\in D$. 
\end{itemize}

\par Let $A= D\cap\{ x> 3/ 4\}$ 
and $U$ a neighborhood of $0$ such that $U\subset R (U) $ and $R (U) $ is on the left of the straight line $x= 1/ 2$. Also suppose that $F (A)\subset U\cup\hat I$. We can change $f$ in $U$ so that it gets differentiability at $p=0$. Let ${\cal U}\subset C^1(\R^2)$ be the set of 
maps $f=R\circ F$ where $R$ and $F$ have the previous properties. 

\begin{prop}\label{unsconst} 
For every map $f\in{\cal U}$ we have that $\W(p)= D\setminus\left (\cup_{ n\geq 1} R^ n (A)\right)$ \end{prop}

\begin{proof} Since $F (D)\subset D$ and $R (D)= D$ we have 
that $f (D)\subset D$ and therefore $\W(p)\subset D$.     
Since the interval $I$ is wandering by $R$ so is the 
angle $\hat I\setminus \{p\}$. 
Then $R^{-n} (\hat I)\subset  f^{-n} (\hat I) $, 
for every $n\geq 0$, 
this is because $F$ is the identity map in $(x,y)$ if $x<1/2$, 
and then $\hat I\subset \W(p)$. 
Also $R^{-n} (\hat I) \subset \W(p)$. 

\par On the other hand $R^n (A) \cap\W(p)=\emptyset$, 
for $n>1$, since $R (A)\cap\W=\emptyset$. 
This is because $R(A)$ does not have any preimage for  $f$ in $D$. 
The points in $D\setminus \cup_{n\in \Z} R^n(\hat I)$ 
have a sequence of preimages by $R$ that converges to $p$, and it 
is easy to show that it is also a sequence of preimages by $f$. 
Then $D\setminus \cup_{n\in \Z} R^n(\hat I)$ is contained in $\W(p)$. 
\end{proof}

\par Now define a family of functions $F_\mu$ so that every one 
coincides on the left of the straight line $x= 3/ 4$ 
satisfying:
\begin{itemize} 
\item $p\notin F_\mu (A) $ for every $\mu< 0$,
\item $p\in F_\mu(A)$ for every $\mu> 0$ and
\item $F_0 (1,0)= p$.
\end{itemize}

In this way, for any negative value of $\mu$ the map $f_\mu$ has no homoclinic orbit associated to $p$ and for positive ones the map has a regular homoclinic orbit associated to $p$. For $\mu=0$ there is no homoclinic orbit associated to $p$. Also, by the previous proposition the unstable set of $p$ is the same for every $\mu$. 

\end{section}

\end{document}